\documentclass[a4paper,notitlepage]{article}

  \usepackage{geometry}
  \usepackage{latexsym}
  \usepackage{amsfonts}
  \usepackage{amsmath}
  \usepackage{graphicx}
  \usepackage{cite}
  \usepackage{color}
  
  \usepackage{layout}
  \usepackage{amssymb}
  \usepackage{amsthm}
  \usepackage{amscd}
  \usepackage{mathrsfs}
  
  \usepackage[english]{babel}
  \usepackage[T1]{fontenc}
  \usepackage[cp1250]{inputenc}
  
  \newtheorem{theorem}{Theorem}[section]
  \newtheorem{example}{Example}[section]
  
  \newtheorem{lemma}[theorem]{Lemma}

  \newcommand{\R}{\mathbb{R}}
  \newcommand{\N}{\mathbb{N}}

  \title{On the existence of homoclinic type solutions of a class\\
  of inhomogenous second order Hamiltonian systems}
  
  \author{\bf Jakub Ciesielski, Joanna Janczewska\\
  \small Faculty of Applied Physics and Mathematics\\
  \small Gdansk University of Technology\\
  \small Narutowicza 11/12, 80-233 Gda\'{n}sk, Poland\\
  \small jakub.ciesielski@pg.edu.pl, joanna.janczewska@pg.edu.pl\\
  \bf Nils Waterstraat\\
  \small School of Mathematics, Statistics and Actuarial Science\\
  \small University of Kent, Canterbury\\
  \small Kent CT2 7NF, England\\
  \small N.Waterstraat@kent.ac.uk\\}
  
  \date{19th September 2017}
  
\begin{document}
  
  \maketitle
  
  \begin{abstract}
    We show the existence of homoclinic type solutions of second order Hamiltonian
    systems of the type $\ddot{q}(t)+\nabla_{q}V(t,q(t))=f(t)$, where $t\in\R$,
    the $C^1$-smooth potential $V\colon\R\times\R^n\to\R$ satisfies a relaxed
    superquadratic growth condition, its gradient is bounded in the time variable,
    and the forcing term $f\colon\R\to\R^n$ is sufficiently small in the space
    of square integrable functions. The idea of our proof is to approximate
    the original system by time-periodic ones, with larger and larger time-periods.
    We prove that the latter systems admit periodic solutions of mountain-pass type,
    and obtain homoclinic type solutions of the original system from them by passing
    to the limit (in the topology of almost uniform convergence) when the periods
    go to infinity.
  \end{abstract}
  
  
  \section{Introduction}
  
  During the past two decades there have been numerous applications of methods
  from the calculus of variations to find periodic, homoclinic and heteroclinic
  solutions for Hamiltonian systems. Many of the striking results that have been
  obtained by variational methods can be found in the well-known monographs
  of Ambrosetti and Coti Zelati \cite{AmC}, Ekeland \cite{Eke}, Hofer and Zehnder
  \cite{HoZ}, Mawhin and Willem \cite{MaW}, as well as in the review articles
  of Rabinowitz \cite{Rab2, Rab3}.\\
  The aim of this paper is to prove the existence of solutions
  of the second order Hamiltonian system
  
  \begin{equation}\label{hs}
    \left
    \{\begin{array}{ll}
      \ddot{q}(t)+\nabla_{q}V(t,q(t))=f(t),\ t\in\R, \\
      \lim\limits_{t\to\pm\infty}q(t)=\lim\limits_{t\to\pm\infty}\dot{q}(t)=0,
    \end{array}
    \right.
  \end{equation}
  where the $C^1$-smooth potential $V\colon\R\times\R^n\to\R$ satisfies
  a relaxed superquadratic growth condition, its gradient $V_{q}\colon\R\times\R^n\to\R^n$
  is uniformly bounded in the time variable on every compact subset of $\R^n$,
  and the norm of the forcing term $f\colon\R\to\R^n$ in the space of square integrable
  functions is smaller than a bound that we state below in our main theorem.\\
  As homoclinic type solutions are global in time, it is reasonable to use
  global methods to find them rather than approaches based on their initial
  value problems. The homogenous systems of \eqref{hs}, i.e. when $f\equiv 0$,
  have been studied extensively under the assumption of superquadratic
  or subquadratic growth of the potential $V(t,q)$ as $|q|\to\infty$.
  Indeed, there are many results on homoclinic solutions for subquadratic
  Hamiltonian systems (cf. e.g. \cite{Rab3}). The first variational results
  for homoclinic solutions of first order Hamiltonian systems with superquadratic
  growth were found by V.\ Coti Zelati, I.\ Ekeland and E.\ S\'{e}r\'{e} in \cite{CES}
  for time-periodic Hamiltonians. Corresponding results for second order Hamiltonian
  systems were obtained in \cite{Rab1} and \cite{CoR}. S.\ Alama and Y.Y.\ Li \cite{AL}
  showed that asymptotic periodicity in time actually suffices to get a homoclinic solution,
  and E.\ Serra, M.\ Tarallo and S.\ Terracini \cite{STT} weakened their periodicity
  condition to almost periodicity in the sense of Bohr. Finally, Hamiltonian systems
  with superquadratic non-periodic potentials were investigated for example
  by  P.\ Montecchiari and M.\ Nolasco \cite{MoN}, A.\ Ambrosetti and M.\ Badiale
  \cite{AmB}, and by the second author of this paper in \cite{Jan1,Jan3,Jan4}.\\
  Our purpose is to generalize Theorem 1.1 of \cite{CJW}, which deals with the existence
  of solutions of the inhomogeneous systems \eqref{hs} under the rather restrictive
  assumption that the potential $V$ is of the special form
  
  \begin{displaymath}
    V(t,q)=-\frac{1}{2}|q|^2+a(t)G(q),
  \end{displaymath}
  where $a\colon\R\to\R$ is a continuous positive bounded function
  and $G\colon\R^n\to\R$ is of class $C^1$ and satisfies the Ambrosetti-Rabinowitz
  superquadratic growth condition. Here, instead, the potential is
  of the more general form
  
  \begin{displaymath}
    V(t,q)=-K(t,q)+W(t,q)
  \end{displaymath}
  with $C^1$-smooth potentials $K$ and $W$ such that
  
  \begin{itemize}
    \item[$(C_1)$] the maps $\nabla_{q}K$ and $\nabla_{q}W$ are uniformly bounded
    in the time variable $t\in\R$ on every compact subset of $\R^n$,
    \item[$(C_2)$] there exist two positive constants $b_1$, $b_2$
    such that for all $t\in\R$ and $q\in\R^n$
    \[ b_{1}|q|^{2}\leq K(t,q)\leq b_{2}|q|^{2}, \]
    \item[$(C_3)$] $K(t,q)\leq (q,\nabla_{q}K(t,q))\leq 2K(t,q)$ for all $t\in\R$ and $q\in\R^n$,
    \item[$(C_4)$] $\nabla_{q}W(t,q)=o(|q|)$ as $|q|\to 0$ uniformly in $t\in\R$,
    \item[$(C_5)$] there is a constant $\mu>2$ such that for all $t\in\R$ and $q\in\R^n\setminus\{0\}$
    \[ 0<\mu W(t,q)\leq (q,\nabla_{q}W(t,q)), \]
    \item[$(C_6)$] $m:=\inf\{W(t,q)\colon t\in\R \ \wedge \ |q|=1\}>0.$
  \end{itemize}
  Here and subsequently, we denote by $(\cdot,\cdot)\colon\R^n\times\R^n\to\R$
  the standard inner product in $\R^n$ and by $|\cdot|$ its induced norm.\\
  Let us point out that under the above assumptions the Hamiltonian system
  \eqref{hs} has the trivial solution when the forcing term $f$ vanishes.
  Therefore it is reasonable to suppose that homoclinic type solutions exist
  when $f$ is sufficiently small. Our main result affirms this hypothesis
  and it also gives an answer to the question how large the forcing term can be.
  
  \begin{theorem}\label{mthm}
    Set $M:=\sup\{W(t,q)\colon t\in\R \ \wedge \ |q|=1\}$ and $\bar{b}_{1}:=\min\{1,2b_{1}\}$.
    Let us assume that $M<\frac{1}{2}\bar{b}_{1}$ and $(C_1)-(C_6)$ are satisfied.
    If the forcing term $f$ is continuous, bounded, and moreover
    
    \begin{equation}\label{force}
      \left(\,\int\limits_{-\infty}\limits^{\infty}|f(t)|^{2}dt\,\right)^\frac{1}{2}
      <\frac{\sqrt{2}}{4}\left(\bar{b}_{1}-2M\right),
    \end{equation}
    then the inhomogenous system \eqref{hs} possesses at least one solution.
  \end{theorem}  
  
  The idea of our proof, which we give in the following second section,
  is to approximate the original system \eqref{hs} by time-periodic ones,
  with larger and larger time-periods. We show that the approximating systems
  admit periodic solutions of mountain-pass type, and obtain a homoclinic type
  solution of the original system from them by passing to the limit
  (in the topology of almost uniform convergence) when the periods go to infinity.
  Finally, we discuss some examples of Theorem \ref{mthm} in Section \ref{Examples}.
  
  \section{Proof of Theorem \ref{mthm}}\label{Proof}
  
  For each $k\in\N$, let $E_k=W^{1,2}_{2k}(\R,\R^n)$ be the Sobolev space
  of $2k$-periodic functions on $\R$ with values in $\R^n$ and the standard norm
  
  \begin{displaymath}
    \|q\|_{E_k}=\left(\,\int\limits^{k}\limits_{-k}\left(|\dot{q}(t)|^{2}
    + |q(t)|^{2}\right)dt\,\right)^{\frac{1}{2}}.
  \end{displaymath}
  We begin with the following estimate that is crucial in the main part
  of our proof below.
  
  \begin{lemma}\label{lemma1}
    For every $\zeta\in\R$
    and $q\in E_k$ we have
    
    \begin{displaymath}
      \int\limits_{-k}\limits^{k}W(t,\zeta q(t))dt
      \geq m|\zeta|^{\mu}\int\limits_{-k}\limits^{k}|q(t)|^{\mu}dt-2km.
    \end{displaymath}
  \end{lemma}
  
  \begin{proof}\,
    Note at first that the assertion is obviously true if $q=0$ or $\zeta=0$.
    Hence we can assume in the rest of the proof that $\zeta\neq 0$ and $q\neq 0$.
    Then it follows from $(C_5)$ that, for every $q\neq 0$ and $t\in\R$,
    the function $z\colon(0,+\infty)\to\R$ defined by
    
    \begin{displaymath}
      z(\zeta)=W\left(t,\frac{q}{\zeta}\right)\zeta^{\mu}
    \end{displaymath}
    is non-increasing. Hence, for every $t\in\R$,
    
    \begin{equation}\label{leq}
      W(t,q)\leq W\left(t,\frac{q}{|q|}\right)|q|^{\mu},\ \textrm{if} \ 0<|q|\leq 1
    \end{equation}
    and
    
    \begin{equation}\label{geq}
      W(t,q)\geq W\left(t,\frac{q}{|q|}\right)|q|^{\mu},\ \textrm{if} \ |q|\geq 1.
    \end{equation}
    We now fix $\zeta\in\R\setminus\left\{0\right\}$, $q\in E_k\setminus\left\{0\right\}$
    and set
    
    \begin{align*}
      A_k&=\left\{t\in[-k,k]\colon |\zeta q(t)|\leq 1\right\},\\
      B_k&=\left\{t\in[-k,k]\colon |\zeta q(t)|\geq 1\right\}.
    \end{align*}
    By \eqref{geq}, we get
    
    \begin{align*}
      \int_{-k}^{k}W(t,\zeta q(t))dt & \geq \int_{B_k}W(t,\zeta q(t))dt
      \geq \int_{B_k}W\left(t,\frac{\zeta q(t)}{|\zeta q(t)|}\right)|\zeta q(t)|^{\mu}dt \\ &
      \geq m \int_{B_k}|\zeta q(t)|^{\mu}dt \geq m\int_{-k}^k|\zeta q(t)|^{\mu}dt
      - m\int_{A_k}|\zeta q(t)|^{\mu}dt \\ &
      \geq m|\zeta |^{\mu}\int_{-k}^k|q(t)|^{\mu}dt - 2km,
    \end{align*}
    which completes the proof.
  \end{proof}
  Further, to prove Theorem \ref{mthm}, we need the following approximative method.
  
  \begin{theorem}[Approximative Method, \cite{Kra}]\label{approx}
    Let $f\colon\R\to\R^n$ be a non-trivial, bounded, continuous
    and square integrable map. Assume that $V\colon\R\times\R^n\to\R$
    is a $C^1$-smooth potential such that $\nabla_{q}V\colon\R\times\R^n\to\R^n$
    is uniformly bounded in $t$ on every compact subset of $\R^n$, i.e.
    
    \begin{displaymath}
      \forall \ L>0 \ \exists \ C>0 \ \forall \ q\in\R^n \ \forall \ t\in\R \ \ 
      |q|\leq L \Rightarrow |\nabla_{q}V(t,q)|\leq C.
    \end{displaymath}
    Suppose that for each $k\in\N$ the boundary value problem
    
    \begin{displaymath}
      \left\{
      \begin{array}{ll}
        \ddot{q}(t)+\nabla_{q}V_{k}(t,q(t))=f_k(t),\\
        q(-k)-q(k)=\dot{q}(-k)-\dot{q}(k)=0,
      \end{array}
      \right.
    \end{displaymath}
    where $f_k\colon\R\to\R^n$ stands for the $2k$-periodic extension
    of $f|_{[-k,k)}$ to $\R$ and $V_k\colon\R\times\R^n\to\R$ denotes
    the $2k$-periodic extension of $V|_{[-k,k)\times\R^n}$ to $\R\times\R^n$,
    has a periodic solution $q_k\in E_k$ and $\{\|q_k\|_{E_k}\}_{k\in\N}$
    is a bounded sequence in $\R$.
    Then there exists a subsequence $\{q_{k_j}\}_{j\in\N}$
    converging in the topology of $C^{2}_{loc}(\R,\R^n)$ to a function
    $q\in W^{1,2}(\R,\R^n)$ which is a solution of
    
    \begin{displaymath}
      \ddot{q}(t)+\nabla_{q}V(t,q(t))=f(t),\ t\in\R.
    \end{displaymath}
  \end{theorem}
  \noindent
  The approximative method was introduced by Paul H.\ Rabinowitz in \cite{Rab1}
  for homogenous second order Hamiltonian systems with a time-periodic potential.
  Later, the second author of this paper extended it to inhomogenous time-periodic
  Hamiltonian systems (see \cite{IzJ} and \cite{Jan2}), and more recently, Robert
  Krawczyk generalized it to the case of aperiodic potentials.\\
  Let us now consider for $k\in\mathbb{N}$ the boundary value problems
  
  \begin{equation}\label{hsk}
    \left\{
    \begin{array}{ll}
      \ddot{q}(t)-\nabla_{q}K_{k}(t,q(t))+\nabla_{q}W_{k}(t,q(t))=f_k(t),\\
      q(-k)-q(k)=\dot{q}(-k)-\dot{q}(k)=0,
    \end{array}
    \right.
  \end{equation}
  where $f_k\colon\R\to\R^n$ stands for the $2k$-periodic
  extension of $f|_{[-k,k)}$ to $\R$, and $K_k\colon\R\times\R^n\to\R$,
  $W_k\colon\R\times\R^n\to\R$ are the $2k$-periodic extensions
  of $K|_{[-k,k)\times\R^n}$ and $W|_{[-k,k)\times\R^n}$ to $\R\times\R^n$.\\
  As we have already mentioned in the introduction, our proof consists of two steps.
  First, we show the existence of  solutions of \eqref{hsk}, and second, we use
  Theorem \ref{approx} to find a solution of \eqref{hs}.\\
  For our first step, let us consider the functionals $I_{k}\colon E_k\to\R$ given by
  
  \begin{equation}\label{action}
    I_k(q)=\int\limits^{k}\limits_{-k}\left(\frac{1}{2}|\dot{q}(t)|^2+K_k(t,q(t))
    -W_k(t,q(t))\right)dt+\int\limits^{k}\limits_{-k}(f_k(t),q(t))dt.
  \end{equation}
  Standard arguments show that $I_k\in C^1(E_k,\R)$, and 
  
  \begin{equation}\label{derivative}
    I'_{k}(q)v=\int\limits^{k}\limits_{-k}\left((\dot{q}(t),\dot{v}(t))
    +(\nabla_{q}K_k(t,q(t))-\nabla_{q}W_k(t,q(t)),v(t))\right)dt
    +\int\limits^{k}\limits_{-k}(f_k(t),v(t))dt.
  \end{equation}
  Moreover, the critical points of the functional $I_k$ are classical $2k$-periodic
  solutions of \eqref{hsk}, and we now show their existence by using the Mountain
  Pass Theorem. Let us recall the latter result before proceeding with our proof.

  \begin{theorem}[Mountain Pass Theorem, \cite{AmR}]\label{Pass}
    Let $E$ be a real Banach space and $I\colon E\to\R$ a $C^1$-smooth
    functional. If $I$ satisfies the following conditions:
    
    \begin{itemize}
      \item[$(i)$] $I(0)=0$,
      \item[$(ii)$] every sequence $\left\{u_j\right\}_{j\in\N}$ in $E$
      such that $\left\{I(u_j)\right\}_{j\in\N}$ is bounded in $\R$
      and $I'(u_j)\to 0$ in $E^{*}$ as $j\to +\infty$
      contains a convergent subsequence (the Palais-Smale condition),
      \item[$(iii)$] there exist constans $\rho,\alpha>0$ such that
      $I|_{\partial B_{\rho}(0)}\geq\alpha$,
      \item[$(iv)$] there exists $e\in E\setminus\bar{B}_{\rho}(0)$
      such that $I(e)\leq 0$,
    \end{itemize}
    where $B_{\rho}(0)$ is the open ball of radius $\rho$ about $0$ in $E$,
    then $I$ possesses a critical value $c\geq\alpha$ given by
    
    \begin{equation}\label{c}
      c=\inf_{g\in\Gamma}\max_{s\in[0,1]}I(g(s)),
    \end{equation}
    where
    
    \begin{displaymath}
      \Gamma=\left\{g\in C([0,1],E)\colon \ g(0)=0,\ g(1)=e \right\}.
    \end{displaymath}
  \end{theorem}
  \noindent
  We now denote by $L^{\infty}_{2k}(\R,\R^n)$ the space of $2k$-periodic
  essentially bounded functions from $\R$ into $\R^n$ equipped with the norm
  
  \begin{displaymath}
    \|q\|_{L^{\infty}_{2k}}=\textrm{ess}\sup\left\{|q(t)|\colon t\in[-k,k]\right\}.
  \end{displaymath}
  It is well known that for each $k\in\N$ and $q\in E_k$ 
  
  \begin{equation}\label{sqrt}
    \|q\|_{L_{2k}^{\infty}}\leq\sqrt{2}\|q\|_{E_k}.
  \end{equation}
  Furthermore, we will write $L^2_{2k}(\R,\R^n)$ for the Hilbert space
  of $2k$-periodic functions on $\R$ with values in $\R^n$ and with the norm
  
  \begin{displaymath}
    \|q\|_{L^{2}_{2k}}
    =\left(\,\int\limits^{k}\limits_{-k}|q(t)|^{2}dt\,\right)^{\frac{1}{2}}.
  \end{displaymath}
  Note that by \eqref{force},
  
  \begin{equation}\label{estimation}
    \|f_k\|_{L^{2}_{2k}}<\frac{\sqrt{2}}{4}\left(\bar{b}_{1}-2M\right).
  \end{equation}
  
  The following lemma shows the existence of a solution of \eqref{hsk}
  and is the main part of the first step of our proof.

  \begin{lemma}\label{lemma2}
    For each $k\in\N$, the functional $I_k$ has a critical value
    of mountain pass type.
  \end{lemma}
  
  \begin{proof}\,
    We let $k\in\N$ be fixed and note at first that it is evident by $(C_2)$ and $(C_5)$
    that $I_k(0)=0$, which shows $(i)$ in Theorem \ref{Pass}.\\
    For checking the Palais-Smale condition $(ii)$, we consider a sequence
    $\{u_j\}_{j\in\N}\subset E_k$ such that $\{I_k(u_j)\}_{j\in\N}$ is bounded in $\R$
    and $I'_k(u_j)\to 0$ in $E^{*}_{k}$ as $j\to\infty$. Then there exists a constant
    $C_k>0$ such that for all $j\in\N$
    
    \begin{equation}\label{estim1}
      |I_k(u_j)|\leq C_k
    \end{equation}
    and
    
    \begin{equation}\label{estim2}
      \|I'_k(u_j)\|_{E^{*}_{k}}\leq C_k.
    \end{equation}
    Now, we will first show that $\{u_j\}_{j\in\N}$ is bounded in the Hilbert space $E_k$.
    Using \eqref{action} and $(C_5)$ we get
    
    \begin{align*}
      2I_k(u_j) & \geq \int_{-k}^{k}\left(|\dot{u}_{j}(t)|^2+2K_k(t,u_j(t))\right)dt
      - \frac{2}{\mu}\int_{-k}^{k}(\nabla_{q}W_k(t,u_j(t)),u_j(t))dt \\ &
      + 2\int_{-k}^{k}(f_k(t),u_j(t))dt.
    \end{align*}
    From \eqref{derivative} and $(C_3)$ it follows that
    
    \begin{align*}
      I'_{k}(u_j)u_j & \leq \int_{-k}^{k}\left(|\dot{u}_{j}(t)|^2+2K_k(t,u_j(t))\right)dt
      - \int_{-k}^{k}(\nabla_{q}W_k(t,u_j(t)),u_j(t))dt \\ &
      + \int_{-k}^{k}(f_k(t),u_j(t))dt.
    \end{align*}
    Thus
    
    \begin{align*}
      2I_k(u_j)-\frac{2}{\mu}I'_{k}(u_j)u_j & \geq
      \left(1-\frac{2}{\mu}\right)\int_{-k}^{k}\left(|\dot{u}_{j}(t)|^2+2K_k(t,u_j(t))\right)dt \\ &
      + \left(2-\frac{2}{\mu}\right)\int_{-k}^{k}(f_k(t),u_j(t))dt,
    \end{align*}
    and by $(C_2)$ we have
    
    \begin{align*}
      2I_k(u_j)-\frac{2}{\mu}I'_{k}(u_j)u_j \geq
      \left(1-\frac{2}{\mu}\right)\bar{b}_{1}\|u_j\|_{E_k}^{2}
      + \left(2-\frac{2}{\mu}\right)\int_{-k}^{k}(f_k(t),u_j(t))dt.
    \end{align*}
    Finally, aplying the H\"{o}lder inequality, as well as \eqref{estimation}, \eqref{estim1}
    and \eqref{estim2}, we obtain
    
    \begin{displaymath}
      \left(1-\frac{2}{\mu}\right)\bar{b}_{1}\|u_j\|_{E_k}^{2}
      -\frac{2C_k}{\mu}\|u_j\|_{E_k}
      -\frac{\sqrt{2}}{4}(\bar{b}_{1}-2M)\left(2-\frac{2}{\mu}\right)\|u_j\|_{E_k}
      -2C_k \leq 0.
    \end{displaymath}
    Since $\mu>2$ we conclude that $\{u_j\}$ is bounded.\\
    Going to a subsequence if necessary, we can assume that there exists a function
    $u\in E_k$ such that $u_j\rightharpoonup u$ weakly in $E_k$ as $j\to +\infty$.
    Hence $u_j\to u$ uniformly on $[-k,k]$, which implies that
    
    \begin{align}\label{conv1}
      \begin{split}
        \left(I'_k(u_j)-I'_k(u)\right)(u_j-u) & \to 0,\\
        \|u_j-u\|_{L^{2}_{2k}} & \to 0
      \end{split}
    \end{align}
    and
    
    \begin{align*}
      \int^{k}_{-k}(\nabla_{q}K_k(t,u_j(t)) & -\nabla_{q}W_k(t,u_j(t)),u_j(t)-u(t))dt\\ &
      - \int^{k}_{-k}(\nabla_{q}K_k(t,u(t))-\nabla_{q}W_k(t,u(t)),u_j(t)-u(t))dt \to 0
    \end{align*}
    as $j\to +\infty$. On the other hand, it is readily seen that
    
    \begin{align*}
      \|\dot{u}_j-\dot{u}\|^{2}_{L^{2}_{2k}} & =(I'_{k}(u_j)-I'_{k}(u))(u_j-u)\\ &
      + \int^{k}_{-k}(\nabla_{q}K_k(t,u_j(t))-\nabla_{q}W_k(t,u_j(t)),u_j(t)-u(t))dt\\ &
      - \int^{k}_{-k}(\nabla_{q}K_k(t,u(t))-\nabla_{q}W_k(t,u(t)),u_j(t)-u(t))dt,
    \end{align*}
    and consequently
    
    \begin{equation}\label{conv2}
      \|\dot{u}_j-\dot{u}\|_{L^{2}_{2k}}\to 0.
    \end{equation}
    By \eqref{conv1} and \eqref{conv2}, we see that $\|u_j-u\|_{E_k}\to 0$, and thus $I_k$
    satisfies the Palais-Smale condition.\\
    To show $(iii)$, we set
    
    \begin{displaymath}
      \varrho=\frac{\sqrt{2}}{2}
    \end{displaymath}
    and assume that $q\in E_{k}$ such that $\|q\|_{E_k}=\varrho$.
    Note that $\|q\|_{L^{\infty}_{2k}}\leq 1$ by \eqref{sqrt}. Thus, we can apply
    \eqref{leq} to obtain
    
    \begin{displaymath}
      \int^{k}_{-k}W(t,q(t))dt \leq
      \int^{k}_{-k}W\left(t,\frac{q(t)}{|q(t)|}\right)|q(t)|^{\mu}dt \leq
      M\int^{k}_{-k}|q(t)|^{2}dt \leq
      M\|q\|^{2}_{E_k} = \frac{1}{2}M.
    \end{displaymath}
    From this, $(C_2)$ and \eqref{force}, we get
    
    \begin{align}\label{alpha}
      \begin{split}
        I_k(q) & \geq
        \frac{1}{2}\bar{b}_{1}\|q\|_{E_k}^{2}-\frac{1}{2}M-\|f_k\|_{L^{2}_{2k}}\|q\|_{E_k} \\
        & \geq \frac{1}{4}(\bar{b}_{1}-2M)-\frac{\sqrt{2}}{2}\|f\|_{L^2} \\
        & = \frac{\sqrt{2}}{2}\left(\frac{\sqrt{2}}{4}(\bar{b}_{1}-2M)-\|f\|_{L^2}\right)
        \equiv \alpha>0
      \end{split}
    \end{align}
    To complete the proof, we have to show $(iv)$, i.e. we need to find $e_k\in E_k$
    such that $\|e_k\|_{E_k}>\rho$ and $I_k(e_k)\leq 0$.
    
    \newpage
    Let
    
    \begin{displaymath}
      \bar{b}_{2}=\max\{1,2b_2\}.
    \end{displaymath}
    Combining \eqref{action} and Lemma \ref{lemma1} gives
    
    \begin{equation}\label{function}
      I_k(\zeta q) \leq \frac{\bar{b}_{2}\zeta^2}{2}\|q\|^{2}_{E_k}
      -m|\zeta|^{\mu}\int^{k}_{-k}|q(t)|^{\mu}dt
      +|\zeta|\cdot\|f_k\|_{L^{2}_{2k}}\|q\|_{E_k}+2km
    \end{equation}
    for all $\zeta\in\R\setminus\{0\}$ and $q\in E_k\setminus\{0\}$.\\
    We now let $Q\in E_1$ be such that $Q\neq 0$ and $Q(-1)=Q(1)=0$.
    It follows from \eqref{function} that $\|\zeta Q\|_{E_1}>\rho$
    and $I_1(\zeta Q)<0$ for $\zeta\in\R\setminus\{0\}$ large enough.
    Hence, if we define $e_1(t)=\zeta Q(t)$ and for each $k\geq 2$,
    
    \begin{equation}\label{vector}
      e_k(t)=\left\{
             \begin{array}{lr}
               e_1(t) & \textrm{for}\, \, t\in[-1,1],\\
               0 & \text{for}\, \, t\in[-k,-1)\cup(1,k],
             \end{array}
             \right.
    \end{equation}
    then $e_k\in E_k$, and $\|e_k\|_{E_k}=\|e_1\|_{E_1}>\rho$ as well as $I_k(e_k)=I_1(e_1)<0$.\\
    In summary, it follows from Theorem \ref{Pass} that the action functional $I_k$
    has a critical value $c_k\geq\alpha$ given by
    
    \begin{equation}\label{critical}
      c_k=\inf_{g\in\Gamma_k}\max_{s\in[0,1]}I_k(g(s)),
    \end{equation}
    where
  
    \begin{displaymath}
      \Gamma_k=\left\{g\in C([0,1],E_k)\colon g(0)=0,\ g(1)=e_k \right\}.
    \end{displaymath}
  \end{proof}

  In what follows, we let $q_k$ be a critical point for the corresponding critical value $c_k$
  that we have found in Lemma \ref{lemma2}. The functions $q_k$, $k\in\N$, are solutions
  of \eqref{hsk} and as second step of our proof of Theorem \ref{mthm}, we now want to apply
  Theorem \ref{approx} to this sequence of functions. 
  
  \begin{lemma}\label{lemma3}
    The sequence $\{\|q_k\|_{E_k}\}_{k\in\N}\subset\R$ is bounded.
  \end{lemma}
  
  \begin{proof}\,
    We set
    
    \begin{displaymath}
      M_0=\max_{s\in[0,1]}I_1(se_1).
    \end{displaymath}
    and conclude from \eqref{vector} and \eqref{critical} that
    
    \begin{equation}\label{M0}
      c_k\leq M_0
    \end{equation}
    for each $k\in\N$. By assumption,
    
    \begin{align*}
      c_k & =I_k(q_k)=I_k(q_k)-\frac{1}{2}I'_{k}(q_k)q_k
      = \int_{-k}^{k}\left(K_k(t,q_k(t))-\frac{1}{2}(\nabla_{q}K_k(t,q_k(t)),q_k(t))\,\right)dt \\
      & +\int_{-k}^{k}\left(\frac{1}{2}(\nabla_{q}W_k(t,q_k(t)),q_k(t))-W_k(t,q_k(t))\,\right)dt
      +\frac{1}{2}\int_{-k}^{k}(f_k(t),q_k(t))dt.
    \end{align*}
    Applying $(C_3)$ and $(C_5)$ we obtain
    
    \begin{displaymath}
      c_k\geq \left(\frac{\mu}{2}-1\right)\int^k_{-k}W_k(t,q_k(t))dt
      + \frac{1}{2}\int^k_{-k}(f_k(t),q_k(t))dt.
    \end{displaymath}
    Furthermore, it follows from \eqref{action} and $(C_2)$ that
    
    \begin{displaymath}
      \int^k_{-k}{W_{k}(t,q_k(t))dt}\geq\frac{1}{2}\bar{b}_{1}\|q_k\|_{E_k}^{2}
      + \int_{-k}^{k}(f_k(t),q_k(t))dt - I_k(q_k).
    \end{displaymath}
    Using that $I_k(q_k)=c_k$, the previous two inequalities give
    
    \begin{displaymath}
      \frac{1}{2}\bar{b}_{1}\|q_k\|_{E_k}^{2}
      -\frac{\mu-1}{\mu-2}\|f_k\|_{L^{2}_{2k}}\|q_k\|_{E_k}
      -\frac{\mu}{\mu-2}c_k \leq 0,
    \end{displaymath}
    which implies by \eqref{force} and \eqref{M0} that
    
    \begin{displaymath}
      \frac{1}{2}\bar{b}_{1}\|q_k\|_{E_k}^{2}
      -\frac{\sqrt{2}}{4}\left(\bar{b}_{1}-2M\right)\frac{\mu-1}{\mu-2}\|q_k\|_{E_k}
      -\frac{\mu}{\mu-2}M_0 \leq 0 .
    \end{displaymath}
    Hence there is $M_1>0$ such that for each $k\in\N$,
    
    \begin{displaymath}
      \|q_k\|_{E_k}\leq M_1.
    \end{displaymath}
  \end{proof}
  
  Now, using Theorem \ref{approx} we see that there exists a solution
  $q\colon\R\to\R^n$ of \eqref{hs} such that $q(t)\to 0$ as $|t|\to\infty$.
  
  All what is left to show for the proof of Theorem \ref{mthm} is that actually
  $\dot{q}(t)\to 0$ as $|t|\to\infty$. This, however, follows from the inequality
  
  \begin{equation}\label{izyjan}
    |\dot{q}(t)|\leq\sqrt{2}\left(\int_{t-\frac{1}{2}}^{t+\frac{1}{2}}
    \left( |\dot{q}(s)|^2 + |\ddot{q}(s)|^2 \right)ds\right)^{\frac{1}{2}},\quad t\in\mathbb{R},
  \end{equation}
  which can be found in \cite{IzJ} (Inequality $(28)$, p.\ $385$).
  Indeed, we just need to note that by \eqref{hs}, $(C_2)$, $(C_4)$ and \eqref{force}
  
  \begin{displaymath}
    \int\limits_{t-\frac{1}{2}}\limits^{t+\frac{1}{2}}|\ddot{q}(s)|^{2}ds\to 0,\quad |t|\to\infty.
  \end{displaymath}
  If now $|t|$ goes to $\infty$ in \eqref{izyjan} we see that $|\dot{q}(t)|\to 0$ as $|t|\to\infty$.
  Consequently, $q$ is a solution of \eqref{hs} and the proof of Theorem \ref{mthm} is complete.

  \section{One-dimensional Examples}\label{Examples}
  
  In this section we present examples for $n=1$ satisfying
  the assumptions of Theorem \ref{mthm}, and the graphs
  of their approximating solutions $q_k$ of \eqref{hsk} for increasing values of $k$.
  
  \begin{example}
    Consider $K\colon\R\times\R\to\R$, $W\colon\R\times\R\to\R$
    and $f\colon\R\to\R$ given by
    
    \begin{displaymath}
      K(t,q)=\frac{t^2+1}{t^2+2}q^2,
    \end{displaymath}
    
    \begin{displaymath}
      W(t,q)=\frac{t^2+12}{3t^2+27}q^4
    \end{displaymath}
    and
    
    \begin{displaymath}
      f(t)=\frac{1}{36}e^{-t^2},
    \end{displaymath}
    where $t,q\in\R$.
    One can easily check that $K,W$ and $f$ satisfy the assumptions
    of Theorem \ref{mthm}.
    The figures \ref{fig:57}-\ref{fig:250} show the graphs
    of numerical solutions $q_k$ of \eqref{hsk}
    for $k = 57, 100, 250$. 
  \end{example}
  
  \begin{example}
    Let $K\colon\R\times\R\to\R$, $W\colon\R\times\R\to\R$
    and $f\colon\R\to\R$ be given by
    
    \begin{displaymath}
      K(t,q)=\left(\frac{1}{8}sin(t)+\frac{1}{8}sin(\sqrt{2}t)+\frac{3}{4}\right)q^2,
    \end{displaymath}
    
    \begin{displaymath}
      W(t,q)=\frac{1}{4}q^4
    \end{displaymath}
    and
    
    \begin{displaymath}
      f(t)=\frac{1}{32}e^{-t^2},
    \end{displaymath}
    where $t,q\in\R$.
    It is immediate that $K,W$ and $f$ satisfy the assumptions of Theorem \ref{mthm}.
    The figures \ref{2fig:10}-\ref{2fig:160} show the graphs of numerical solutions
    $q_k$ of \eqref{hsk} for $k = 10, 40, 160$.
  \end{example}
  
   \begin{example}
    Consider $K\colon\R\times\R\to\R$, $W\colon\R\times\R\to\R$
    and $f\colon\R\to\R$ given by
    
    \begin{displaymath}
      K(t,q)=q^2,
    \end{displaymath}
    
    \begin{displaymath}
      W(t,q)=\frac{10}{33}q^4\left(arctg^2\left(\frac{q^2}{t^2+1}\right)+1\right)
    \end{displaymath}
    and
    
    \begin{displaymath}
      f(t)=\frac{1+t^2}{10}e^{-t^2},
    \end{displaymath}
    where $t,q\in\R$.
    Again, it is readily seen that $K,W$ and $f$ satisfy the assumptions
    of Theorem \ref{mthm}.
    The figures \ref{3fig:100}-\ref{3fig:180} show the graphs
    of numerical solutions $q_k$ of \eqref{hsk}
    for $k = 100, 140, 180$. 
  \end{example}
  
      
  \begin{figure}[t]
    \centering
    \includegraphics[width=10cm]{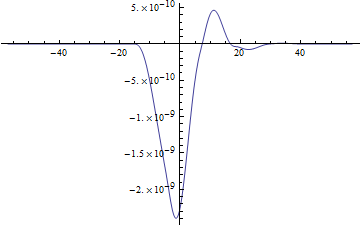}
    \caption{A numerical solution of \eqref{hsk} for $k=57$ in Example 1}
    \label{fig:57}
  \end{figure}
  
  \begin{figure}[b]
    \centering
    \includegraphics[width=10cm]{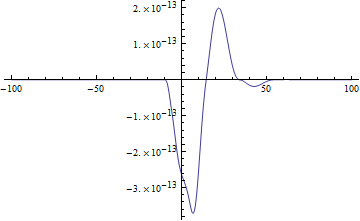}
    \caption{A numerical solution of \eqref{hsk} for $k=100$ in Example 1}
    \label{fig:100}
  \end{figure}
  
  \begin{figure}[p]
    \centering
    \includegraphics[width=10cm]{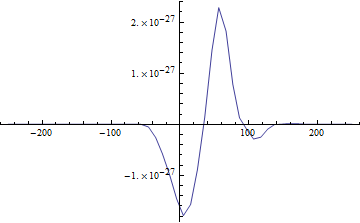}
    \caption{A numerical solution of \eqref{hsk} for $k=250$ in Example 1}
    \label{fig:250}
  \end{figure}
  
  \begin{figure}[p]
    \centering
    \includegraphics[width=10cm]{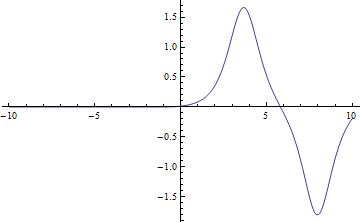}
    \caption{A numerical solution of \eqref{hsk} for $k=10$ in Example 2}
    \label{2fig:10}
  \end{figure}
  
  \begin{figure}[p]
    \centering
    \includegraphics[width=10cm]{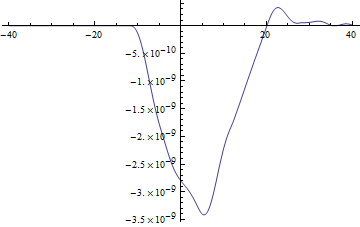}
    \caption{A numerical solution of \eqref{hsk} for $k=40$ in Example 2}
    \label{2fig:40}
  \end{figure}
  
  \begin{figure}[p]
    \centering
    \includegraphics[width=10cm]{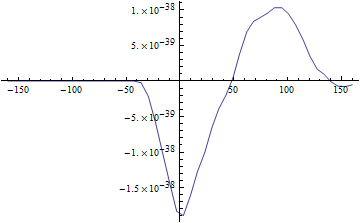}
    \caption{A numerical solution of \eqref{hsk} for $k=160$ in Example 2}
    \label{2fig:160}
  \end{figure}
    
  \begin{figure}[p]
    \centering
    \includegraphics[width=10cm]{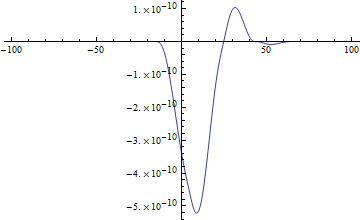}
    \caption{A numerical solution of \eqref{hsk} for $k=100$ in Example 3}
    \label{3fig:100}
  \end{figure}
  
  \begin{figure}[p]
    \centering
    \includegraphics[width=10cm]{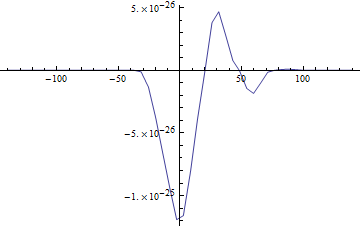}
    \caption{A numerical solution of \eqref{hsk} for $k=140$ in Example 3}
    \label{3fig:140}
  \end{figure}
  
  \begin{figure}[t]
    \centering
    \includegraphics[width=10cm]{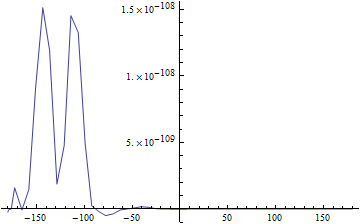}
    \caption{A numerical solution of \eqref{hsk} for $k=180$ in Example 3}
    \label{3fig:180}
  \end{figure}
  
  
  \newpage
  
  \bibstyle

  \end{document}